\documentclass[reqno, 12pt]{amsart}
\usepackage[margin=1in]{geometry}
\usepackage{a4wide, amsfonts, amssymb, amsmath, amsthm, mathrsfs, enumerate, enumitem, setspace, url}
\usepackage[utf8]{inputenc}
\usepackage{graphicx}
\usepackage[all]{xy}
\usepackage{mathtools}
\usepackage{color}
\usepackage{xcolor}
  \definecolor{luhblue}{RGB}{0,80,155} 
  \definecolor{darkblue}{rgb}{0,0,0.5}
  \definecolor{darkred}{rgb}{0.5,0,0}
  \definecolor{darkgreen}{rgb}{0,0.5,0}
\usepackage{hyperref} 
  \hypersetup{linkcolor=darkblue,filecolor=darkgreen,urlcolor=darkred,citecolor=darkblue}

\usepackage{csquotes}
\usepackage{tikzscale}
\usepackage{tikz}
\usepackage{tikz-cd}
\usetikzlibrary{babel,matrix,arrows,decorations.pathmorphing}
\tikzset{commutative diagrams/.cd}

\newtheorem{theorem}{Theorem}[section]

\newtheorem{lemma}[theorem]{Lemma}

\newtheorem{proposition}[theorem]{Proposition}

\theoremstyle{definition}
\newtheorem{definition}[theorem]{Definition}

\theoremstyle{remark}
\newtheorem{remark}[theorem]{Remark}
\newtheorem{example}[theorem]{Example}
\newtheorem*{acknowledgements}{Acknowledgements}

\renewcommand{\(}{\left(}
\renewcommand{\)}{\right)}

\newcommand{\ZZ}{\mathbb{Z}}
\newcommand{\QQ}{\mathbb{Q}}
\newcommand{\Zp}{\mathbb{Z}_p}

\newcommand{\Qp}{\mathbb{Q}_p}

\newcommand{\RR}{\mathbb{R}}

\newcommand{\Fp}{\mathbb{F}_p}

\renewcommand{\AA}{\mathbb{A}}
\newcommand{\PP}{\mathbb{P}}

\newcommand{\sD}{\mathcal{D}}

\newcommand{\sG}{\mathcal{G}}
\newcommand{\sM}{\mathcal{M}}

\newcommand{\sO}{\mathcal{O}}

\newcommand{\sU}{\mathcal{U}}
\newcommand{\sV}{\mathcal{V}}
\newcommand{\sX}{\mathcal{X}}

\newcommand{\Domega}{D_{\uomega}}
\newcommand{\sDomega}{\sD_{\uomega}}

\newcommand{\spec}{\mathrm{Spec}}

\DeclareMathOperator{\HH}{H}

\DeclareMathOperator{\Br}{Br}
\DeclareMathOperator{\inv}{inv}

\DeclareMathOperator{\Cl}{Cl}
\DeclareMathOperator{\GL}{GL}
\DeclareMathOperator{\SL}{SL}

\renewcommand{\val}{v}
\DeclareMathOperator{\valp}{v_{\it p}}

\DeclareMathOperator{\et}{\acute{e}t}
\renewcommand{\epsilon}{\varepsilon}

\DeclareMathOperator{\uomega}{\underline{\omega}}
\DeclareMathOperator{\red}{red}
\DeclareMathOperator{\fin}{fin}
\DeclareMathOperator{\cc}{C}
\DeclareMathOperator{\dd}{D}
\DeclareMathOperator{\str}{st}
\DeclareMathOperator{\ns}{ns}

\usepackage{blindtext}
\numberwithin{equation}{section}

\begin{document}
\onehalfspacing
\title[Semi-integral points on Markoff orbifold pairs]{Local-global principles for semi-integral points on Markoff orbifold pairs}

\author{Vladimir Mitankin}
  \address{Vladimir Mitankin, Institute of Mathematics and Informatics, Bulgarian Academy of Sciences, Acad. G. Bonchev St. bl.8, 1113 Sofia, Bulgaria}

  \address{Institut für Algebra, Zahlentheorie und Diskrete Mathematik,
Leibniz Universität Hannover, Welfengarten 1, 30167 Hannover, Germany}
  \email{v.mitankin@math.bas.bg}

\author{Justin Uhlemann}
  \address{Justin Uhlemann, Mathematical Institute, Utrecht University, Hans Freudenthalgebouw, Budapestlaan 6, 3584 CD Utrecht, The Netherlands}
  \email{j.a.uhlemann@uu.nl}

\date{\today}

\begin{abstract}
  We study local-global principles for semi-integral points on orbifold pairs of Markoff type. In particular, we analyse when these orbifold pairs satisfy weak weak approximation, weak approximation and strong approximation off a finite set of places. We show that Markoff orbifold pairs satisfy the semi-integral Hasse principle and we measure how often such orbifold pairs have strict semi-integral points but the corresponding Markoff surface lacks integral points.
\end{abstract}  

\maketitle
\setcounter{tocdepth}{1}
\tableofcontents

\section{Introduction}
\label{sec:intro}
Let $m \neq 0,4$ be an integer. A smooth affine Markoff surface $U_m$ over $\QQ$ and its $\ZZ$-model $\sU_m$, given by the same equation, are defined by
\begin{equation}
  \label{eqn:Um}
  U_m: \quad u_1^2 + u_2^2 + u_3^2 - u_1u_2u_3 = m.
\end{equation}
Ghosh and Sarnak's pioneering work \cite{Ghosh2017IntegralPO} conjectured the amount of $m$ with $\sU_m(\ZZ) = \emptyset$ but no local obstructions to that, which sparked a lot of interest in the number theory community. Soon after, two independent studies \cite{LM21} and \cite{colxu_2020} examined that conjecture together with the density of integral points on $\sU_m$. The behaviour of integral points on singular and generalised versions of Markoff surfaces have been investigated in \cite{Zag82}, \cite{Sil90}, \cite{Bar05}, \cite{Chen21}, \cite{CS23}, \cite{Dao23}, \cite{KSS23}, \cite{Dao24}, \cite{Mis24}. 

This work pursues a different direction by exploring for the first time Campana and Darmon points, collectively abbreviated semi-integral, on orbifold pairs defined by \eqref{eqn:Um}. Semi-integrality (Definition~\ref{def:semi-integral}) is a concept generalising integral and rational points and at the same time interpolating between them. We conduct a comprehensive study of the semi-integral local-global principles, as introduced in \cite[\S2]{MNS22}, for Markoff orbifold pairs and compare their status with their rational and integral counterparts. This introduces the investigation of pairs with multiple irreducible components of the boundary.

\subsection*{Local-global principles}
\label{subsec:local-global}
To put things in perspective, consider the smooth compactification $X_m \subset \PP_{\QQ}^3$ of $U_m$, with a $\ZZ$-model $\sX_m$ defined by the same equation, given by
\begin{equation}
  \label{eqn:Xm}
  X_m: \quad x_0(x_1^2 + x_2^2 + x_3^2) - x_1x_2x_3 = mx_0^3.
\end{equation}
The existence of a rational point $(0:0:1:1) \in X_m(\QQ)$ implies that $X_m$ is unirational \cite{Kol02} and that it satisfies weak weak approximation \cite{SD01}. In fact, $X_m$ is rational, and thus satisfies weak approximation, if and only if $m - 4 \in \QQ(\sqrt{m})^{*2}$ by \cite[Lem.~3.3.]{LM21}, which is equivalent to $m - 4 \in \QQ^{*2}$ (see Remark~\ref{rem:m}). The reverse implication also holds.

\begin{theorem} 
  \label{thm:wa-Xm}
  The variety $X_m$ satisfies weak approximation for $\QQ$-rational points if and only if $X_m$ is rational.
\end{theorem}

To set-up the notion of semi-integrality, let $D = X_m \setminus U_m$. Fixing weights $\omega_i \in \ZZ_{\ge 2} \cup \{\infty\}$ for the irreducible components $D_i = \{ x_0 = x_i = 0\}$, $i = 1, 2, 3$ of $D$ defines a $\QQ$-divisor $\Domega = \sum_{i = 1}^3(1 - 1/\omega_i)D_i$ that gives rise to a Campana orbifold structure $(X_m,\Domega)$. If $\sD_i$ is the Zariski closure of $D_i$ in $\sX_m$ for $i = 1, 2, 3$ and $\sDomega = \sum_{i = 1}^3(1 - 1/\omega_i)\sD_i$, the $\ZZ$-model $(\sX_m,\sDomega)$ of $(X_m,\Domega)$ carries the data of semi-integral points on the pair (Definition~\ref{def:semi-integral}). Selecting $\omega_1 = \omega_2 = \omega_3 = \infty$ or an empty boundary divisor recovers integral points on $\sU_m$ and rational points on $X_m$, respectively. Our results go beyond the classical setting and consider the case where at least one $\omega_i < \infty$. The discrepancy in the behaviour of semi-integral and rational points is observed in the next result. 

\begin{theorem}
  \label{thm:wa}
  The following hold.
  \begin{enumerate}[label=\emph{(\roman*)}]
    \item Assume that at least two of $\omega_1, \omega_2, \omega_3$ are finite. Then each pair $(\sX_m, \sDomega)$ fails weak weak approximation for semi-integral points.
    \item Assume that exactly one of $\omega_1, \omega_2, \omega_3$ is finite. If $m - 4 \in \QQ^{*2}$, then $(\sX_m, \sDomega)$ satisfies weak approximation. On the other hand, if $m - 4 \notin \QQ^{*2}$, then 
    \begin{enumerate}[label=\emph{(\alph*)}]
      \item $(\sX_m, \sDomega)$ fails Campana weak approximation; 
      \item $(\sX_m, \sDomega)$ fails Darmon weak approximation if the finite weight is even.
    \end{enumerate}
  \end{enumerate}
\end{theorem}

Theorem~\ref{thm:wa}(i) follows from a result about weak approximation for general orbifold pairs, namely Theorem~\ref{thm:wwa-general-pairs}. To the best of our knowledge this is the first result on local-global principles for Campana orbifolds with multiple irreducible components of the boundary. 

\begin{remark}
  \label{rem:adeles}
  Currently, there is no uniform notion for the semi-integral adelic space across the literature. We follow the convention of \cite{MNS22}. The proofs of Theorems~\ref{thm:wa} and \ref{thm:wwa-general-pairs} do not adapt to the other notions of semi-integral adeles, e.g the one given in \cite{pieropan2021}.
\end{remark}

Theorem~\ref{thm:wa} implies that, because of the boundary, strong approximation off any finite set fails if at least two of $\omega_1, \omega_2, \omega_3$ are finite. The next result exhibits this failure through semi-integral points on $U_m$ in the case of an arbitrary number of finite weights.

\begin{theorem}
  \label{thm:sa}
  Let $S(m) = \{p \in \Omega_{\QQ} \ : \ m - 4 \in \Qp^{*2} \}$. If $m - 4 \notin \QQ^{*2}$ and some $\omega_i$, $i = 1, 2, 3$ is finite, then $(\sX_m, \sDomega)$ fails strong approximation off any finite set $T \subset S(m)$.
\end{theorem}


\begin{remark} 
  \label{rem:hp}
  Assume that not all of $\omega_1, \omega_2, \omega_3$ are infinite. Then the family of orbifold pairs $(\sX_m, \sDomega)$ satisfies the semi-integral Hasse principle as $(0:0:1:1) \in X_m(\QQ)$ is semi-integral on $(\sX_m,\sDomega)$ by Proposition~\ref{prop:semi-integral-criterion}. Thus for any choice of weights neither Campana nor Darmon points on $(\sX_m, \sDomega)$ behave in the way that $\sU_m(\ZZ)$ does, as $\sU_m(\ZZ)$ is infinitely often empty without local obstruction to that as shown in \cite{Ghosh2017IntegralPO}, \cite{LM21} and \cite{colxu_2020}.
\end{remark}

\subsection*{Strict points}
\label{subsec:strict}
Integral points are better approximated by strict semi-integral points $(\sX_m, \sDomega)_{\str}^*(\ZZ) = (\sX_m, \sDomega)^*(\ZZ) \cap U_m(\QQ)$ (Definitions~\ref{def:semi-integral} and \ref{def:str-nstr}), where $\ast = \cc$ to indicate Campana points or $\ast = \dd$ for Darmon points. For example $(\sX_m, \sDomega)_{\str}^*(\ZZ)$ detects failures of strong approximation featuring in the integral case \cite[Thm.~1.1]{LM21}, e.g. Theorem~\ref{thm:sa}. 

A powerful tool to study local-global principles is the Brauer--Manin obstruction. Its semi-integral version has been defined in \cite[\S3]{MNS22} along with its relation to its rational and integral counterparts. As is evident from Remark~\ref{rem:hp} such obstructions to the Hasse principle are not present in the family $(\sX_m, \sDomega)$ as its non-strict part vanishes. We perform a complete analysis of the values of the local invariant maps for the relevant Brauer elements in \S\ref{sec:Markoff} to examine this obstruction for strict semi-integral points.

\begin{theorem}
  \label{thm:BMO-hp}
  Assume that $\omega_i < \infty$ for some $i = 1, 2, 3$. Then no algebraic Brauer--Manin obstruction to the Hasse principle for strict points on $(\sX_m, \sDomega)$ is present. Moreover, if
  \begin{equation}
    \label{eqn:star}
    \tag{$\ast$}
    -(m - 4) \in \QQ^2 \implies \frac{\sqrt{m} + \sqrt{m - 4}}{2} \notin \QQ(\sqrt{m}, \sqrt{m - 4})^2
  \end{equation} 
  is met, then there is no Brauer--Manin obstruction to the Hasse principle for strict semi-integral points on $(\sX_m, \sDomega)$.
\end{theorem}

The condition \eqref{eqn:star} is not very restrictive on $m$, in fact $\#\{|m| \le B \ : \ \text{\eqref{eqn:star} fails} \} = O(B^{1/4})$ by \cite[Prop~5.2]{LM21}.

The lack of a Brauer--Manin obstruction does not necessarily imply the existence of strict semi-integral points. By \cite[Thm~1.2(ii) and Prop.~6.1]{Ghosh2017IntegralPO} asymptotically $7/12$ of the total number of Markoff orbifold pairs trivially have a strict semi-integral point, as $\sU_m(\ZZ) \neq \emptyset$. For the remaining $5/12$ values of $m$ we have $\sU_m(\ZZ) = \emptyset$ because no $2$-adic or $3$-adic integral points exist. However, it may still be the case that there are global semi-integral points, as explained in the paragraph after Question~1.7 in \cite{MNS22}. It is natural to focus on those remaining $m$ with empty integral adeles. For any real $B \ge 1$ let 
\[
 N(B) = \# \{ |m| \le B \ : \ \sU_m(\AA_\ZZ) = \emptyset \text{ but } (\sX_m,\sDomega)_{\str}^*(\ZZ) \neq \emptyset \}.
\]

A similar quantity was first examined for quadratic orbifold pairs in \cite{MNS22}, leading to the best known lower bound for the number of Hasse failures for integral points on affine diagonal quadrics. The problem of understanding how often local-global principle hold in families has garnered a lot of attention lately. Several papers \cite{BBL16}, \cite{LRS22}, \cite{LM24} explore it for rational points in general families of algebraic varieties. In \cite{BB14a}, \cite{Rom19} Châtelet surfaces are analysed, \cite{BB14b} targets coflasque tori, \cite{BL20} studies Erdős-Straus surfaces, \cite{MS22} deals with quartic del Pezzo surfaces and \cite{GLN22}, \cite{San23} focuses on certain classes of K3 surfaces.

Local-global principles for integral points in families are much less understood. Well-known birational invariants of smooth proper varieties are non-invariant for affine varieties. The set of integral points depends on the choice of integral model, and no general theory predicts Hasse failures \cite[Prop~5.20]{LM21}. Concrete families are studied by \cite{Dic66}, Mordell \cite{Mor42} and Heath-Brown \cite{HB92} who examine sums of three cubes. As for Markoff surfaces, obstructions to the existence of integral points have been found for other families of log K3s such as affine cubic surfaces \cite{CTW12}, \cite{U23}, \cite{LMU23} and affine del Pezzo surfaces of degree $4$ and $5$ \cite{JS17}, \cite{L23}. Few other investigations of degree 2 and 3 surfaces feature in \cite{Mit17}, \cite{San22}, \cite{Wan21}.

Clearly, $N(B) = O(B)$. We use the theory of quadratic forms and the explicit description of strict semi-integral points in Proposition~\ref{prop:semi-integral-criterion-2-inf} to deduce an almost sharp lower bound for the number of pairs which non-trivially satisfy the strict semi-integral Hasse principle.

\begin{theorem}
  \label{thm:count}
  Assume that one of $\omega_i$ is finite. Then 
  \[
    N(B) \gg \frac{B}{(\log B)^{1/2}}.
  \]
\end{theorem}

The quantification of $N(B)$ provides a solid step forward in the direction of understanding failures of local-global principles for strict and general semi-integral points. The proof of Theorem~\ref{thm:count} is illustrative in the sense that we construct an explicit subfamily of Markoff orbifold pairs with strict semi-integral points but no integral adelic points. Two other such families are given in Theorems~\ref{thm:existencepcases} and \ref{thm:existence-CTXW}. We hope that a general theory explaining the failures of semi-integral local-global principles may serve as a remedy for the lack of understanding how integral points behave in families.

\subsection*{Outline} 
In \S\ref{sec:background} we recall the notion of semi-integrality along with local-global principles for semi-integral points and the Brauer--Manin obstruction to them. We then turn attention in \S\ref{sec:Markoff} to the particular case of interest in this work, namely Markoff orbifold pairs, and translate the existing theory to that setting. We conclude \S\ref{sec:Markoff} with the proofs of the qualitative results in Theorems~\ref{thm:wa-Xm}, \ref{thm:wa}, \ref{thm:sa} and \ref{thm:BMO-hp}. Finally, \S\ref{sec:quad-forms} is dedicated to the existence result in Proposition~\ref{prop:st-semi-integral-points}, the proof of the counting problem set in Theorem~\ref{thm:count} and the exhibition of other subfamilies of Markoff orbifold pairs with strict semi-integral points but no integral points.

\subsection*{Notation}
We write $k$ for a number field, $\sO_k$ for its ring of integers and $\Omega_k$ for the set of places of $k$. We fix $\Omega_\infty \subset \Omega_k$ for the set of all Archimedean places. If $S \subset \Omega_k$ is a finite subset containing $\Omega_\infty$, $\sO_S$ stands for the ring of $S$-integers of $k$, recovering $\sO_k$ if $S = \Omega_\infty$.

We shall use $U$ for a smooth affine variety over $k$ (a separated scheme of finite type), $X$ will be a smooth projective variety and $D$ a divisor on $X$. If $X$ is the compactification of $U$ inside the projective space, $D$ will be its boundary divisor so that $U = X \setminus D$. Integral models of $X$, $U$, $D$ over $\sO_S$ will be denoted by $\sX$, $\sU$, $\sD$. 

Whenever orbifold pairs related to Markoff surfaces are under discussion, $m \in \ZZ \setminus \{0, 4\}$ will be assumed and $\uomega = (\omega_1,\omega_2,\omega_3)$ with $\omega_i \in \ZZ_{\ge 2} \cup \{\infty\}$, $i = 1, 2, 3$ will be fixed. The notation $X_m$, $U_m$, $D_i$ and $\Domega$ defined in the introduction and $\sX_m$, $\sU_m$, $\sD_i$ and $\sDomega$ for their integral models over $\ZZ$, respectively, as in the introduction will be in force.

\begin{acknowledgements}
  We would like to thank Ulrich Derenthal and Florian Wilsch for several fruitful discussions, Francesco Campagna for improvements to section \ref{sec:quad-forms}, and Daniel Loughran and Marta Pieropan for helpful comments. We thank the anonymous referee for their careful reading and helpful suggestions and comments that have improved this paper. While working on this paper the first author was supported by a Horizon Europe MSCA postdoctoral fellowship 101151205 -- GIANT funded by the European Union, by scientific program ``Enhancing the Research Capacity in Mathematical Sciences (PIKOM)'', No. DO1-67/05.05.2022 of the Ministry of Education and Science of Bulgaria, and by grant DE 1646/4-2 of the Deutsche Forschungsgemeinschaft. The second named author is supported by the NWO grant VI.Vidi.213.019. For the purpose of open access, a CC BY public copyright license is applied to any Author Accepted Manuscript version arising from this submission.
\end{acknowledgements}

\section{Local-global principles for semi-integral points}
\label{sec:background}
In this section, we recall the notions of orbifold pairs defined by Campana \cite{Cam04}, \cite{Cam05}, \cite{Cam11a}, \cite{Cam11b}, \cite{Cam15}, semi-integral points, semi-integral local-global principles and their Brauer--Manin obstruction as defined in \cite{MNS22}. 

\begin{definition}
  A (Campana) orbifold or an orbifold pair over $k$ is a pair $(X, \Domega)$, where $X$ is a smooth proper variety over $k$ and $\Domega$ is a $k$-rational $\QQ$-divisor on $X$ of the shape
  \[
  \Domega = \sum_{i\in I}\left(1-\frac{1}{\omega_i}\right)D_i.
  \]
  Here $I$ is a finite (possibly empty) indexing set of distinct irreducible effective (Weil) divisors $D_i$ of $X$, $\uomega = (\omega_i)_{i \in I}$ such that $\omega_i \in \ZZ_{\ge 2} \cup \{\infty\}$. We call $\omega_i$ the weight of $D_i$. The support of $\Domega$ is $D_{\red} = \cup_{i \in I} D_i$. We say that $(X, D)$ is smooth if $D_{\red}$ is a strict normal crossings divisor.
\end{definition}
\begin{remark}
  Note that, in comparison with \cite{MNS22}, we require a finite indexing set of irreducible divisors for $\Domega$ and so instead of $\omega_i \ge 1$, we have $\omega_i \ge 2$. This is only a cosmetic change. The notation in \cite{MNS22} was motivated from the general theory viewpoint and by the assumption that one starts with a projective variety $X$ and selects a boundary divisor $D$ on it. The notion of integrality depends on the boundary divisor. The current set-up is more suitable for generalising questions about integral points on affine varieties, in which case the boundary has to stay fixed but the weights may vary.
\end{remark}

\begin{definition}
 Let $(X, \Domega)$ be an orbifold pair. Fix a finite set $S \subset \Omega_k$ containing $\Omega_{\infty}$. Consider a flat, proper scheme $\mathcal X$ of finite type over $\sO_S$, with an isomorphism $\sX_{(0)} \cong X$ for the generic fibre of $\sX$. Denote by $\sDomega$ the divisor $\sum_{i\in I}(1-1/\omega_i)\mathcal D_i$, where $\sD_i$ for each $i \in I$ is the Zariski closure of $D_i$ in $\sX$ along that isomorphism. The pair $(\mathcal X,\sD_{\uomega})$ is called an \emph{$\sO_S$-model} of $(X,D_{\underline \omega})$.
\end{definition}

\begin{remark}
 Notice that, in contrast to the construction given by the authors in \cite{pieropan2021}, we also allow $\mathcal X$ to be irregular. Allowing irregularity gives us more freedom in choosing the set $S$ of finite places which define our $\sO_S$-model. This also turned out to be useful in the setting of counting Campana points of bounded height, as discussed in \cite[\S3]{pieropan2023}.
\end{remark}

\subsection{Semi-integral points}
Let $(X,\Domega)$ be an orbifold pair with a $\sO_S$-model $(\sX,\sDomega)$. Any $M \in X(k)$ lifts to a point $M_v \in \sX(\sO_v)$ for all $v \in \Omega_K \setminus S$ by the valuative criterion for properness \cite[Theorem 7.3.8.]{EGAII}.
For any $\sD_i \subset \sX$, we want to quantify the position of $M_v \in \sX(\sO_v)$ with respect to $\sD_i$ at all places $v \in \Omega_K \setminus S$, by assigning it a value in the following way. Consider the fibre product of the two morphisms given by the closed immersion $\sD_i \longrightarrow \sX$ and by the local point $M_v: \spec(\sO_v) \longrightarrow \sX$.
Since closed immersions are stable under base change \cite[3.11.a]{hartshorne_1977} we know that $\spec(\sO_v) \times_{\sX} \sD_i$ is a closed subscheme of $\spec(\sO_v)$. Every closed subscheme of $\spec(\sO_v)$ is of the shape $\spec(\sO_v/I)$ for an ideal $I \subset \sO_v$. As $\sO_v$ is a discrete valuation ring with uniformiser $\pi_v$, say, every non-zero ideal is of the shape $I = (\pi_v^n)$ for some $n \in \ZZ_{\ge 0}$.

\begin{definition} \label{def:intersectionmult}
  Let $(X, \Domega)$ be an orbifold pair with a $\sO_S$-model $(\sX,\sDomega)$. If $M_v \in \sX(\sO_v)$ with $v \in \Omega_k \setminus S$, define 
  \[
    n_v(\sD_i, M_v) \coloneq
    \begin{cases}
       \infty &\text{if } I=(0) \text{ (or equivalently, } M_v \in D_i(k_v)),\\
       n &\text{if } I = (\pi_v^n).
    \end{cases} 
  \]
  We call $n_v(\sD_i, M_v)$ the \emph{intersection multiplicity} of $M_v$ with $\sD_i$.
\end{definition}

As we shall soon see, if $\omega_i = \infty$ for some $D_i$ in $D_{\red}$, any semi-integral point must be away from that $D_i$, i.e. its intersection multiplicity with $\sD_i$ must be 0. This motivates the need of the following definition. 

\begin{definition}
  Set $D_{\inf} := \bigcup_{\omega_i = \infty} D_i$ and $\sD_{\inf} := \bigcup_{\omega_i = \infty} \sD_i$. Analogously, for each $i \in I$, set $D_{i, \inf}:= D_i \cap D_{\inf}$ and $\sD_{i ,\inf}:= \sD_i \cap \sD_{\inf}$ as well as $D_{i,\fin} := D_i \setminus D_{i, \inf}$ and $\sD_{i, \fin}:= \sD_i \setminus \sD_{i,\inf}$. Finally, set $D_{\fin} := D_{\red} \setminus D_{\inf}$ and $\sD_{\fin} := \sD_{\red} \setminus \bigcup_{\omega_i = \infty} \sD_i$.
\end{definition}

We are now in position to define Campana points and Darmon points. We collectively abbreviate them as semi-integral points, meaning that one may replace the words ``semi-integral" with either ''Campana" or ``Darmon" to obtain valid statements that are independent of the semi-integral notion chosen.

\begin{definition}
  \label{def:semi-integral}
  Let $(X, \Domega)$ be an orbifold pair with a $\sO_S$-model $(\sX,\sDomega)$. For any $v \in \Omega_k$ define the set of $v$-adic (or local) semi-integral points, denoted $(\sX,\sDomega)^*(\sO_v)$, by
  \[
    (\sX,\sDomega)^*(\sO_v) = 
    \begin{cases}
      (X \setminus D_{\inf})(k_v) &\text{if } v \in S,\\
      \{M_v \in (\sX \setminus \sD_{\inf})(\sO_v) \ : \ n_v(\sD_i, M_v) \text{ admissible if } \omega_i \neq \infty \} &\text{if } v \notin S,
    \end{cases}
  \]
  where the additional admissibility condition on $n_v(\sD_i, M_v)$ for $v \in \Omega_k \setminus S$ is given by:
  \begin{itemize}
    \item $n_v(\sD_i, M_v) \in \ZZ_{\ge \omega_i} \cup \{0, \infty\}$ for Campana points;
    \item $n_v(\sD_i, M_v) \in \omega_i\ZZ_{\ge 0} \cup \{\infty\}$ for Darmon points. 
  \end{itemize}
  We say that $M \in X(k)$ \emph{is a (global) semi-integral point on $(\sX,\sDomega)$} if $M$ is a $v$-adic semi-integral on $(\mathcal X,\sD_{\uomega})$ for all $v \in \Omega_K$. We denote the set of (global) semi-integral points by $(\sX, \sDomega)^*(\sO_S)$.
\end{definition}

If we want to specify which of the two notions we are working with, we shall substitute $*=\cc$ for Campana points and $*=\dd$ for Darmon points.

\begin{remark}
  In view of the above definition, any global semi-integral point $M$ must be an $S$-integral point of $\sX \setminus \sD_{\inf}$ and hence $(\sX,\sDomega)^*(\sO_S) \subset (\sX \setminus \sD_{\inf})(\sO_S)$.
\end{remark}

\begin{remark}
  From the definition of Campana, and of Darmon points, these two notions of semi-integral points only differ at how they intersect with the orbifold divisor. In fact, it is clear from the definition that
  \[
    (\sX,\sDomega)^{\dd}(\sO_S) \subset (\sX,\sDomega)^{\cc}(\sO_S) \quad \text{and} \quad
    (\sX,\sDomega)^{\dd}(\sO_v) \subset (\sX,\sDomega)^{\cc}(\sO_v).
  \] 
  Furthermore, the connection between Campana points and $m$-full numbers, and between Darmon points and $m$-th powers is now apparent in view of the following (for which a detailed explanation can be found at \cite[Example 1]{pieropan2023}). If an open subset $\sU \subset \sX$ is given such that $M_v \in \sU(\sO_v)$ and $\sD_i$ is locally defined by a rational function $f$, which is regular on $\sU$, then $n_v(\sD_i, M_v)= \val_v(f(M_v))$, where $\val_v$ denotes the $v$-adic valuation of $\sO_v$. 
\end{remark}

\begin{remark}
A generalisation of the semi-integrality notion, termed $\sM$-points, was recently introduced in \cite{Moe24}. The main idea behind it is to allow more flexibility for the set of admissible $n_v(\sD_i, M_v)$ in Definition~\ref{def:semi-integral}. This new notion features some interesting examples of pairs whose global points are of particular interest to the number theory community.
\end{remark}

\subsection{Adelic space and local-global principles} 
Given a Campana orbifold $(X, \Domega)$, set $U = X \setminus D_{\red}$ and $\sU = \sX \setminus \cup_{i \in I} \sD_i$. As explained in \cite{MNS22}, to define an adelic space in a ``minimal'' way, we need to distinguish between semi-integral points away from the boundary divisor and those of them lying on $D_{\fin}$. Another reason to partition $(\sX, \sDomega)$ into a $U$ piece and a $D_{\fin}$ piece, and study those pieces separately, is that points on the $D_{\fin}$ piece come from a lower dimensional variety and generally may behave differently, especially when it comes to local-global principles. This is illustrated Theorem~\ref{thm:wwa-general-pairs}.

\begin{definition}
  \label{def:str-nstr}
  Let $(X, \Domega)$ be an orbifold pair with a $\sO_S$-model $(\sX,\sDomega)$. We define by
  \[
    \begin{split}
      (\sX, \sDomega)_{\str}^{*}(\sO_S) &\coloneq (\sX,\sDomega)^{*}(\sO_S) \cap U(k), \\
      (\sX, \sDomega)_{\ns}^{*}(\sO_S) &\coloneq (\sX,\sDomega)^{*}(\sO_S) \cap \sD_{\fin}(\sO_S),
    \end{split}
  \]
  the sets of \emph{strict} and of \emph{non-strict} semi-integral points on $(\sX,\sDomega)$, respectively. Thus the set of global semi-integral points may be partitioned as the disjoint union
  \[
    (\sX, \sDomega)^{*}(\sO_S) = (\sX, \sDomega)_{\str}^{*}(\sO_S) \coprod (\sX, \sDomega)_{\ns}^{*}(\sO_S).
  \]
  We define the strict $(\sX, \sDomega)_{\str}^{*}(\sO_v) = (\sX,\sDomega)(\sO_v) \cap U(k_v)$ and non-strict $(\sX, \sDomega)_{\ns}^{*}(\sO_S) = (\sX,\sDomega)(\sO_v) \cap \sD_{\fin}(\sO_v)$ local points in an analogous way, therefore giving a partitioning for local points
   \[
    (\sX, \sDomega)^{*}(\sO_v) = (\sX, \sDomega)_{\str}^{*}(\sO_v) \coprod (\sX, \sDomega)_{\ns}^{*}(\sO_v).
  \]
\end{definition}

The next natural object to look at is the semi-integral adelic space. We shall work with the definitions in \cite[\S2]{MNS22} which are well-suited to examine obstructions to local-global principles. We recall the definitions of the semi-integral adelic space, of the Hasse principle, of weak and of strong approximation below.

\begin{definition} 
  \label{def:-semi-integral-adeles}
  Given a finite set $T \subset \Omega_k$, we define the sets of \emph{strict $T$-adelic semi-integral points} $(\sX,\sDomega)_{\str}^*(\AA_{k, S}^T)$ and \emph{non-strict $T$-adelic semi-integral points} $(\sX,\sDomega)_{\ns}^*(\AA_{k, S}^T)$ of $(\sX,\sDomega)$ by
  \[
    \begin{split}
      \(\sX,\sDomega\)_{\str}^*\(\AA_{k, S}^T\)
      &:= \prod_{v \in S \setminus (S \cap T)} U(k_v) \times \sideset{}{'}\prod_{v \in \Omega_K \setminus (S\cup T)}\( \(\sX,\sDomega\)^*\(\sO_v\) \cap U(k_v),\sU\(\sO_v\)\) 
      , \\
      \(\sX,\sDomega\)_{\ns}^*\(\AA_{k, S}^T\)
      &:=\bigcup_{\substack{i \in I \\ \omega_i \neq \infty}}
      \left( \prod_{v \in S \setminus (S \cap T)} D_{i, \fin}(k_v) 
      \times \prod_{v \in \Omega_k \setminus (S\cup T)} \(\sX,\sDomega\)^*\(\sO_v\) \cap \sD_{i, \fin}\(\sO_v\) 
      \right).
    \end{split}
  \]
  We equip $(\sX,\sDomega)_{\str}^*(\AA_{k, S}^T)$ with the restricted product topology, while $(\sX,\sDomega)_{\ns}^*\(\AA_{k, S}^T\)$ is considered as a subset of $\prod_{v \in S \setminus (S \cap T)}D_{\fin}(k_v) \times \prod_{v \in \Omega_k \setminus (S\cup T)} \sD_{\fin}(\sO_v)$ whose product topology it inherits. We define the set of \emph{$T$-adelic semi-integral points} as the disjoint union
  \[
    \(\sX,\sDomega\)^*\(\AA_{k, S}^T\) 
    := \(\sX,\sDomega\)_{\str}^*\(\AA_{k, S}^T\) \coprod \(\sX,\sDomega\)_{\ns}^*\(\AA_{k, S}^T\),
  \]
  endowed with the coproduct topology, which we call the \emph{adelic topology} on $(\sX,\sDomega)^*(\AA^T_{k, S})$.

  We define the \emph{adelic semi-integral points} to be the $T$-adelic semi-integral points for $T = \emptyset$ and denote them by $(\sX,\sDomega)^*(\AA_{k, S})$. We define the sets of \emph{strict} and \emph{non-strict adelic semi-integral points} analogously. We omit $S$ when it is clear from context.
\end{definition}

\begin{definition} 
  \label{def:semi-integral-HP}
  Let $\mathscr{F}$ be a collection of $\sO_S$-models of Campana orbifolds. We say that $\mathscr{F}$ satisfies the \emph{Hasse principle for semi-integral points} if the following implication holds for all $(\sX,\sDomega) \in \mathscr{F}$:
  \[
    \(\sX, \sDomega\)^*\(\AA_{k, S}\) \neq \emptyset 
    \implies \(\sX,\sDomega\)^*\(\sO_S\) \neq \emptyset.
  \]
\end{definition}

In a similar fashion, the Hasse principle may be defined for strict or non-strict semi-integral points by restricting to the desired component of the adelic space. In fact, for the main example of orbifold pairs in this work the strict semi-integral Hasse principle will turn out to be of a more interesting nature.

Recall that Definition~\ref{def:-semi-integral-adeles} ensures that $(\sX, \sDomega)^*(\AA_{k, S \cup T}^T) = (\sX, \sDomega)^*(\AA_{k, S}^T)$. This leads to the following definitions of weak and of strong approximation.

\begin{definition} 
  \label{def:WA-SA}
  Let $(\sX,\sDomega)$ be an $\sO_S$-model of a Campana orbifold $(X, \Domega)$, and let $T$ be a finite set of places of $k$.
  \begin{enumerate}[label=(\roman*)]
    \item We say that $(\sX,\sDomega)$ satisfies \emph{weak approximation (for semi-integral points) off $T$} (abbreviated \emph{WA off $T$}) if $(\sX,\sDomega)^*(\sO_{S \cup T})$ is dense in $(\sX, \sDomega)^*(\AA_{k, S}^T)$ for the product topology. We say that $(\sX,\sDomega)$ satisfies \emph{weak weak approximation (for semi-integral points)} (abbreviated \emph{WWA}) if it satisfies WA off some finite set of places $T$. We say that $(\sX,\sDomega)$ satisfies \emph{weak approximation (for semi-integral points)} (abbreviated \emph{WA}) if it satisfies WA off $T = \emptyset$.
    \item We say that $(\sX, \sDomega)$ satisfies \emph{strong approximation (for semi-integral points) off $T$} (abbreviated \emph{SA} off $T$) if $(\sX, \sDomega)^*(\sO_{S \cup T})$ is dense in $(\sX, \sDomega)^*(\AA_{k, S}^T)$ for the adelic topology. We say that $(\sX,\sDomega)$ satisfies \emph{strong approximation (for semi-integral points)} (abbreviated \emph{SA}) if it satisfies SA off $T = \emptyset$.
  \end{enumerate}
  We extend the above definitions to strict and non-strict semi-integral points by restricting the strict or non-strict adelic space, respectively.
\end{definition}

We conclude with the status of weak weak approximation. Comparing the following theorem to \cite[Theorem 1.3]{NS24} shows how the different definitions of adeles can lead to different results on the approximation properties of general orbifold pairs. We refer to \cite[Appendix A]{Moe24} for a detailed comparison.

\begin{theorem}
  \label{thm:wwa-general-pairs}
   Let $(\sX,\sDomega)$ be an $\sO_S$-model of a Campana orbifold $(X, \Domega)$. Assume that there exists at least two distinct irreducible $D_i, D_j \subset D_{\red}$ of finite weights such that $D_{i, \fin}$ and $D_{j, \fin}$ are geometrically integral. Then $(\sX,\sDomega)$ fails WWA.
\end{theorem}

\begin{proof}
  Let $D_i, D_j$ be as in the statement. Select $(M_v) \in (\sX,\sDomega)_{\ns}^*(\AA_{k, S})$ such that there are infinitely many $M_v \in \sD_{i, \fin}(\sO_v) \setminus \bigcup_{k \neq i} \sD_{k, \fin}(\sO_v)$ and  infinitely many $M_v \in \sD_{j, \fin}(\sO_v) \setminus \bigcup_{k \neq j} \sD_{k, \fin}(\sO_v)$. Such $(M_v)$ exists as $D_{i, \fin}$ and $D_{j, \fin}$ are geometrically integral, their respective intersections with $D_{\fin}$ minus the respective component are of positive codimension and thus the Lang--Weil estimates and Hensel’s lemma imply that the above sets are non-empty for all sufficiently large $v$. Chose a finite $T \subset \Omega_K$. Then clearly, $(M_v)_{v \notin T}$ has infinitely many components coming from $D_{i, \fin}$ and infinitely many coming from $D_{j, \fin}$ as above and all of them are away from their respective intersections with $D_{\fin}$ minus the respective component. It follows from Definition~\ref{def:WA-SA} that $(M_v)_{v \notin T}$ may only be approximated by a non-strict semi-integral global point. Let $S' \subset \Omega_k \setminus T$ be finite and such that there are $v', v'' \in S'$ for which $M_{v'} \in \sD_{i, \fin}(\sO_{v'}) \setminus \bigcup_{k \neq i} \sD_{k, \fin}(\sO_{v'})$ and $M_{v''} \in \sD_{j, \fin}(\sO_{v''}) \setminus \bigcup_{k \neq j} \sD_{k, \fin}(\sO_{v''})$. Chose $\varepsilon > 0$ less than the minimum of the $v'$-adic distance between $M_{v'}$ and $\bigcup_{k \neq i} \sD_{k, \fin}(\sO_{v'})$ and less than the $v''$-adic distance between $M_{v''}$ and $\bigcup_{k \neq j} \sD_{k, \fin}(\sO_{v''})$. If there were $M \in (\sX,\sDomega)_{\ns}^*(\sO_S)$ for which $|M - M_v|_{v'} < \varepsilon$, then as $(\sX,\sDomega)_{\ns}^*(\sO_S) \subset (\sX,\sDomega)_{\ns}^*(\sO_{v'})$ one must have $M \in (\sD_{i, \fin}(\sO_v) \setminus \bigcup_{k \neq i} \sD_{k, \fin}(\sO_v)) \cap (\sX,\sDomega)_{\ns}^*(\sO_S) = \sD_{i, \fin}(\sO_S) \setminus \bigcup_{k \neq i} \sD_{k, \fin}(\sO_S)$. Similarly, as $|M - M_w|_{v''} < \varepsilon$, then as any such $M$ has to lie in $\sD_{j, \fin}(\sO_S) \setminus \bigcup_{k \neq j} \sD_{k, \fin}(\sO_S)$. But the intersection of $\sD_{i, \fin}(\sO_S) \setminus \bigcup_{k \neq i} \sD_{k, \fin}(\sO_S)$ and $\sD_{j, \fin}(\sO_S) \setminus \bigcup_{k \neq j} \sD_{k, \fin}(\sO_S)$ is clearly empty as $i \neq j$. Therefore WA off $T$ fails. As $T$ is arbitrary, we conclude that WWA fails for $(\sX,\sDomega)$, which completes the proof.
\end{proof}
 
\subsection{The semi-integral Brauer--Manin obstruction}
\label{subsec:bmo}
Let $V$ be a smooth but not necessarily proper variety over $k$ with an $\sO_S$-model $\sV$. We begin by recalling the classical definition of the Brauer--Manin obstruction for rational points given by Manin \cite{manin}. Its version for integral points was defined by Colliot-Thélène and Xu \cite{intbmo} in a similar fashion. Our main reference for the classical theory is \cite{colliotsko_2021}.

\begin{definition}
  The Brauer group of $V$ is the second étale cohomology group $\Br V := \HH^{2}_{\et}(V,\mathbb{G}_{m})$. Let $\Br_1 V:=\ker(\Br V \rightarrow \Br V_{\bar{k}})$. The group $\Br_1 V$ will be referred to as the \emph{algebraic Brauer group of $V$}, while $\Br V/\Br_{1} V$ is the \emph{transcendental Brauer group of $V$}.
\end{definition} 

There is a canonical injective homomorphism $\inv_v: \Br k_v \rightarrow \QQ/\ZZ$ for any $v \in \Omega_k$ \cite[Def.~13.1.7]{colliotsko_2021}, whose image is $\frac{1}{2}\ZZ/\ZZ \subset \QQ/\ZZ$ if $v$ is a real place, $0$ if it is a complex place and is an isomorphism if $v$ is finite \cite[Thm.~13.1.8]{colliotsko_2021}. 

Given $\alpha \in \Br V$, there exists a finite set of places $S_{\alpha} \subseteq \Omega_k$, such that $\inv_v \alpha (\cdot) : V (k_v) \rightarrow \QQ/\ZZ$ vanishes on $\sV(\sO_v)$ for all $v \not\in S_{\alpha}$ \cite[Prop.~13.3.1]{colliotsko_2021}. This shows that the Brauer--Manin pairing, as given below, is well defined.
\[
  \begin{split}
    V(\AA_k) \times \Br V &\longrightarrow \QQ/\ZZ, \\ 
    \((M_v),\alpha\) &\longmapsto \sum_{v \in \Omega_k} \inv_v(\alpha(M_v)).
  \end{split}
\] 

The Brauer--Manin set, denoted $V(\AA_k)^{\Br}$ or $V(\AA_k)^{\Br V}$ if one wants to emphasise on the Brauer group used, is defined as the left kernel of that paring. Since the following diagram commutes:
\[
  \begin{tikzcd}
    & V(k) \arrow[r,hook] \arrow[d,"\alpha(\cdot)"] & V(\AA_k) \arrow[d,"\alpha(\cdot)"] & & \\
    0 \arrow[r] & \Br k \arrow[r] & \bigoplus\limits_{v \in \Omega_K} \Br k_v \arrow[r,"\sum_v \inv_v"] & \QQ/\ZZ \arrow[r] & 0,
  \end{tikzcd}
\]
where exactness of the bottom row is implied by the Albert--Brauer--Hasse--Noether theorem and class field theory, there is a chain of inclusions $V(k) \subseteq V (\AA_k)^{\Br} \subseteq V(\AA_k)$. It gives an obstruction to the existence of $k$-rational points on $V$ and to their density in $V(\AA_k)$. If $B \subset \Br V$, define the intermediate obstruction set $V (\AA_k)^{B}$ as the set of adeles that pair to zero with all elements of $B$. An algebraic Brauer--Manin obstruction refers to selecting $B = \Br_1 V$ and similarly for a transcendental Brauer--Manin obstruction.

We continue with the Brauer--Manin obstruction for semi-integral points as introduced in \cite[\S3]{MNS22}. Recall that for a Campana orbifold $(X, \Domega)$ we have set $U = X \setminus D_{\red}$.

\begin{definition} 
  \label{def:SIBMO}
  Define the strict semi-integral Brauer--Manin set $(\sX, \sDomega)_{\str}^*(\AA_{k, S})^{\Br} = (\sX, \sDomega)_{\str}^*(\AA_{k, S})^{\Br U}$ as the right kernel of the Brauer--Manin pairing $\Br U \times U(\AA_k) \rightarrow \QQ/\ZZ$ when restricted to strict semi-integral points. Define the non-strict semi-integral Brauer--Manin set $\(\sX, \sDomega\)_{\ns}^*\(\AA_{k, S}\)^{\Br}$ as the preimage of $\bigcup_{\omega_i \neq  \infty} D_{i, \fin}(\AA_k)^{\Br D_{i, \fin}}$ under the natural inclusion $\(\sX, \sD\)_{\ns}^*\(\AA_{k, S}\)\hookrightarrow \bigcup_{\omega_i \neq \infty} D_{i, \fin}(\AA_k)$. Then define the semi-integral Brauer--Manin set as the disjoint union
  \[
    \(\sX, \sD\)^*\(\AA_{k, S}\)^{\Br}
    := \(\sX, \sD\)_{\str}^*\(\AA_{k, S}\)^{\Br} \coprod \(\sX, \sD\)_{\ns}^*\(\AA_{k, S}\)^{\Br}.  
  \]
\end{definition}

To study local-global principles, it will be useful to work with the projections of the various Brauer--Manin sets to the $T$-adeles.

\begin{definition} \label{def:Br-T-adeles}
  We define $(\sX, \sD)^*(\AA_{k, S \cup T}^T)^{\Br}$ to be the projection of $(\sX, \sD)^*(\AA_{k, S \cup T})^{\Br}$ to $(\sX, \sD)^*(\AA_{k, S\cup T}^T)$ and similarly for strict, non-strict, integral and rational points.
\end{definition}

We shall be particularly interested in intermediate obstruction sets for strict points. Given a subset $B \subset \Br U$, we can define an intermediate obstruction set $(\sX,\sD)_{\str}^*(\AA_{k, S})^B$ by restricting the Brauer--Manin pairing to only $B$ for the strict points. On taking projections away from $T$ as above, we get a sequence of inclusions 
\[
  \(\sX, \sDomega\)_{\str}^*\(\sO_{S \cup T}\) \subset \(\sX,\sDomega\)_{\str}^*\(\AA_{k, S \cup T}^T\)^{\Br} \subset \(\sX, \sDomega\)_{\str}^*\(\AA_{k, S \cup T}^T\)^B  \subset \(\sX, \sDomega\)_{\str}^*\(\AA_{k, S}^T\).
\] 

Therefore the Brauer--Manin obstruction and its different partial forms can obstruct the local-global principles for semi-integral points, as well as for strict and non-strict semi-integral points. This is explained in detail in the next definitions following \cite[\S3]{MNS22}. 

\begin{definition} 
  \label{def:BMO-HP}
  Let $(X, \Domega)$ be a Campana orbifold over a $k$ with $\sO_S$-model $(\sX,\sDomega)$. There is a \emph{Brauer--Manin obstruction to the semi-integral Hasse principle} if 
  \[
    (\sX,\sDomega)^*(\AA_{k, S}) \neq \emptyset \quad \text{but} \quad (\sX,\sDomega)^*(\AA_{k, S})^{\Br} = \emptyset.
  \] 
  Otherwise, there is no Brauer--Manin obstruction to the semi-integral Hasse principle. A Brauer--Manin obstruction to the Hasse principle for strict points and for non-strict points is defined analogously by restricting to the strict adeles and to the non-strict adeles, respectively.
\end{definition}

For a finite subset $T \subset \Omega_K$ let $\overline{(\sX, \sDomega)_{\str}^*(\AA_{k, S \cup T}^T)^{\Br}}$ and $\overline{(\sX, \sDomega)_{\ns}^*(\AA_{k, S \cup T}^T)^{\Br}}$ denote the closures of $(\sX, \sDomega)_{\str}^*(\AA_{k, S \cup T}^T)^{\Br}$ and $(\sX, \sDomega)_{\ns}^*(\AA_{k, S \cup T}^T)^{\Br}$ in the product topology, respectively. Let also
\[
  \overline{(\sX, \sDomega)^*(\AA_{k, S \cup T}^T)^{\Br}} =
  \overline{(\sX, \sDomega)_{\str}^*(\AA_{k, S \cup T}^T)^{\Br}} \coprod \overline{(\sX, \sDomega)_{\ns}^*(\AA_{k, S \cup T}^T)^{\Br}}.
\] 
We may now recall the Brauer--Manin obstruction to semi-integral weak approximation and to semi-integral strong approximation.

\begin{definition}
  \label{def:BMO-WA-SA}
  Let $(X,\Domega)$ be a Campana orbifold over $k$ with $\sO_S$-model $(\sX,\sDomega)$. 
  \begin{enumerate}[label=(\roman*)]
    \item We say that there is a \emph{Brauer--Manin obstruction to weak approximation off $T$ for semi-integral points} if 
    \[
      \overline{(\sX, \sDomega)^*(\AA_{k, S \cup T}^T)^{\Br}} \neq (\sX, \sDomega)^*(\AA_{k, S}^T).
    \]
    \item We say that there is a \emph{Brauer--Manin obstruction to strong approximation off $T$ for semi-integral points} if 
    \[
      (\sX, \sDomega)^*(\AA_{k, S \cup T}^T)^{\Br} \neq (\sX, \sDomega)^*(\AA_{k, S}^T).
    \]
  \end{enumerate}
  We shall omit ``off $T$'' in each definition above if $T = \emptyset$. Alternatively, if the above assumptions are not satisfied we say that there is no Brauer--Manin obstruction to weak or to strong approximation for semi-integral points, respectively. A Brauer--Manin obstruction to weak weak approximation as well as for weak and for strong approximation for strict and for non-strict semi-integral points is defined in a similar way by restricting the above definitions to strict or non-strict points, respectively.
\end{definition}

\section{Markoff orbifold pairs}
\label{sec:Markoff}
We now turn our attention to the central objects of interest in this paper, namely the Markoff orbifold pairs $(X_m, \Domega)$ associated to \eqref{eqn:Um} with a $\ZZ$-model $(\sX_m, \sDomega)$, as defined in the introduction. For us $S = \{\infty\} \subset \Omega_{\QQ}$ and thus $\sO_S = \ZZ$. 
\subsection*{Semi-integral points}
The next proposition gives an explicit arithmetic description of the set of semi-integral points on $(\sX_m, \sDomega)$.

\begin{proposition}
  \label{prop:semi-integral-criterion}
  Let $p$ be a finite prime. If $M_p = (x_0:x_1:x_2:x_3) \in \sX_m(\Zp) \subset \PP^3(\Zp)$ with $x_0,x_1,x_2,x_3$ not all divisible by $p$, then $n_p(\sD_i, M_p) = \min\{\valp(x_0),\valp(x_i)\}$ for $i = 1,2,3$. In particular,
  \begin{enumerate}[label=\emph{(\roman*)}]
    \item $M_p$ is a local Campana point on $(\sX_m, \sDomega)$ if
    \[
      \min\{\valp(x_0),\valp(x_i)\} \in \ZZ_{\geq \omega_i} \cup \{0, \infty\}, \quad  i=1,2,3;
    \] 
    \item $M_p$ is a local Darmon point on $(\sX_m, \sDomega)$ if
    \[
      \min\{\valp(x_0),\valp(x_i)\} \in \omega_i\ZZ_{\ge 0} \cup \{\infty\}, \quad  i=1,2,3.
    \]
  \end{enumerate}
\end{proposition}

\begin{proof}
  Chose coordinates $\sX_m \subset \PP^3_\ZZ(y_0,y_1,y_2,y_3)$ and consider one of the three components of $\sDomega$, $\sD_1=\{y_0=y_1=0\}$ say. If $M_p \in D_1(\Qp)$, then on one hand $x_0 = x_1 = 0$, while on the other $n_p(\sD_1, M_p) = \infty$, which confirms the claim. Assume now that $M_p \notin D_1(\Qp)$, and let us compute the local intersection multiplicity from Definition \ref{def:intersectionmult}. We may apply the base change $\sX_{m,\ZZ_p}$ since the question is of a local nature: 
  \[
    \begin{tikzcd}
      {M_p \times_{\sX_{m,\Zp}} \sD_1} \arrow[d] \arrow[r]
      & {\mathrm{Proj} \( \frac{\Zp[y_0,y_1,y_2,y_3]}{(y_0,y_1)} \)} \arrow[d] \\
      {\mathrm{Spec} \( \frac{\Zp[y_0,y_1,y_2,y_3]}{(y_0-x_0,y_1-x_1,y_2-x_2,y_3-x_3)}\)}\arrow[r] 
      & {\sX_{m,\Zp}}           
  \end{tikzcd}
  \]
  Following \cite[Example 1]{pieropan2023}, by our assumptions on $M_p$ we can find $a_0,a_1,a_2,a_3\in \ZZ_p$ such that $a_0x_0+a_1x_1+a_2x_2+a_3x_3 = 1$. Write $l(y_0,y_1,y_2,y_3)$ for the associated linear form $a_0y_0+a_1y_1+a_2y_2+a_3y_3$. Notice that the image of $M_p$ is contained in the affine open $\sX_{m,\Zp} \setminus \{l = 0\}$. Computing the fibre product then shows that $M_p \times_{\sX_{m,\Zp}} \sD_1$ is isomorphic to
    \[
     \spec{\(\frac{\Zp[y_0,y_1,y_2,y_3]}{(y_0,y_1,y_0-x_0,y_1-x_1,y_2-x_2,y_3-x_3,l(y_0,y_1,y_2,y_3) - 1)}\)} \cong \spec{\(\frac{\Zp}{(x_0,x_1)}\)}.
    \]
    The ideal $(x_0,x_1)$ in $\Zp$ is generated by $p^k$, where $k$ equals the minimum of $\{\valp(x_0),\valp(x_1)\}$, and therefore we have $n_p(\sD_1,M_p)= \min\{\valp(x_0),\valp(x_1)\}$. The same argument applies to $\sD_2$ and $\sD_3$, which shows the claim.
\end{proof}

We continue with an important proposition, which describes the structure of the set of local semi-integral points which are not locally integral.

\begin{proposition}
  \label{prop:semi-integral-criterion-2-inf}
  Let $p$ be a finite prime and assume that $\omega_i$ is finite, while $\omega_j$ and $\omega_k$ are infinite, where $i, j, k \in \{1, 2, 3\}$ are distinct. Then $(\sX_m, \sDomega)_{\str}^*(\Zp) \setminus \sU(\Zp) \neq \emptyset$ and each point $M_p = (x_0, x_1, x_2, x_3)$ inside that set satisfies $\valp(x_j) = \valp(x_k) = 0$ and the following:
  \begin{enumerate}[label=\emph{(\roman*)}]
    \item if $p \equiv 1 \bmod 4$, then $\valp(x_0) \le \valp(x_i)$;
    \item if $p \equiv 3 \bmod 4$, then $\valp(x_0) = \valp(x_i)$; 
    \item if $p = 2$, then $\mathrm{v}_2(x_0) = \mathrm{v}_2(x_i) - 1$.
  \end{enumerate}
\end{proposition}

\begin{proof}
  Without loss of generality we may assume that $\omega_1$ is finite and $\omega_2 = \omega_3 = \infty$. To see that $M_p \in (\sX_m, \sDomega)_{\str}^*(\Zp) \setminus \sU(\Zp)$ exists, fix $x_0 = p^{\omega_1}$, $x_1 = y_1 p^{\omega_1}$ with $y_1 \in \Zp$, and $x_2 = x_3 \in \Zp^{\times}$. Dividing \eqref{eqn:Xm} through by $p^{\omega_1}$ and reducing mod $p$ gives 
  \[
    2x_2^2 - y_1x_2^2 \equiv 0 \bmod p, 
  \]
  which is clearly soluble with $y_1 \equiv 2 \bmod p$. This is a smooth $\Fp$-point on the variety obtained from the division of \eqref{eqn:Xm} by $p^{\omega_1}$, as its partial derivative with respect to $y_1$ does not vanish. Such a $\Fp$-point lifts to a $\Zp$-point by Hensel's lemma and clearly extends to $M_p = (p^{\omega_1}, y_1p^{\omega_1}, x_2, x_3) \in \sX_m(\Zp)$. Moreover, as $x_0, x_2, x_3$ have been chosen carefully, $M_p$ is semi-integral by Proposition~\ref{prop:semi-integral-criterion} and does not belong to $\sU_m(\Zp)$.

  By definition, each $M_p \in (\sX_m, \sDomega)_{\str}^*(\Zp) \setminus \sU(\Zp)$ must satisfy $\valp(x_0) > 0$. Therefore, $\valp(x_0) \ge \omega_1$ and $\valp(x_2) = \valp(x_3) = 0$ by Proposition~\ref{prop:semi-integral-criterion}. Write $x_0 = y_0 p^{\valp(x_0)}$ and $x_1 = y_1 p^{\valp(x_1)}$. It is now clear that $\valp(x_0) \le \valp(x_1)$, otherwise the reduction of $X_m \bmod p^{\valp(x_1) + 1}$ would imply that $y_1x_2x_3 \equiv 0 \bmod p$, which is a contradiction. As $-1 \in \Fp^{*2}$ if and only if $p \equiv 1 \bmod 4$, we conclude also that $p \nmid y_1$ for $p \equiv 3 \bmod 4$.

  Assume that $p = 2$ and $\mathrm{v}_2(x_0) = \mathrm{v}_2(x_1)$. Dividing through $2^{\mathrm{v}_2(x_1)}$ and reducing $X_m \bmod 2$ implies that 
  \[
    0 \equiv y_0(x_2^2 + x_3^2) \equiv y_1x_2x_3 \equiv 1 \bmod 2, 
  \]
  which is clearly a contradiction. On the other hand, if $\mathrm{v}_2(x_0) < \mathrm{v}_2(x_1)$, we may write $x_1 = y_1x_0$ with $y_1 \in \ZZ_2$. Dividing through $x_0$ and reducing mod 4 then gives
  \[
    x_2^2 - y_1x_2x_3 + x_3^2 \equiv 0 \bmod 4.
  \]
  As $x_2, x_3 \in \ZZ_2^{\times}$ this is only possible if $\mathrm{v}_2(y_1) = 1$. This completes the proof.
\end{proof}

\begin{remark}
  \label{rem:inclusion-and-weights}
  Let $\underline{\omega} = (\omega_1, \infty, \infty)$ and $\underline{\omega}' = (\omega_1, \omega_3, \omega_3)$ be given with $\omega_1 < \infty$. Then, by definition $(\sX_m, \sDomega)_{\str}^*(\Zp) \subseteq (\sX_m, \sD_{\underline{\omega}'})_{\str}^*(\Zp)$ and $(\sX_m, \sDomega)_{\ns}^*(\Zp) \subseteq (\sX_m, \sD_{\underline{\omega}'})_{\ns}^*(\Zp)$.
\end{remark}

\subsection*{Values of the local invariant maps}
Recall that, if $[\QQ(\sqrt{m}, \sqrt{m - 4}) : \QQ] = 4$, then $\Br U_m/\Br \QQ = \langle \alpha_{1,-}, \alpha_{2,-}, \alpha_{3,-} \rangle$ by \cite[Prop.~4.5]{LM21}. Moreover, $\Br X_m / \Br \QQ$ is generated by $\alpha$ as explained in \cite[Lem.~3.2]{LM21}, and the following explicit representations as elements of $\Br U_m$ are valid
\[
  \begin{split}
    &\alpha_{i,-} = (x_i/x_0 - 2, m - 4), \quad i = 1, 2, 3, \\
    &\alpha = (x_1^2/x_0^2 - 4, m - 4) = (x_2^2/x_0^2 - 4, m - 4) =(x_3^2/x_0^2 - 4, m - 4). \\
  \end{split}
\]
Finally, $\alpha$ and the $\alpha_{i, -}$ are linked by
\begin{equation}
  \label{eqn:alpha-alpha_i}
  \alpha = \alpha_{1, -} + \alpha_{2, -} + \alpha_{3, -}.
\end{equation}

\begin{lemma}
  \label{lem:star-equivalent-degree4}
  The conditions $m, m - 4, m(m-4) \notin \QQ^{*2}$ and $[\QQ(\sqrt{m}, \sqrt{m - 4}) : \QQ] = 4$ are equivalent whenever $m \neq 0$. 
\end{lemma}

\begin{proof}
  Assume that $m, m - 4, m(m-4) \notin \QQ^{*2}$. To see that $[\QQ(\sqrt{m}, \sqrt{m - 4}) : \QQ] = 4$ it suffices to verify that $m - 4$ is not a square of $\QQ(\sqrt{m})^*$. If $m - 4$ was an element of $\QQ(\sqrt{m})^{*2}$, then $\sqrt{m - 4} = a + b\sqrt{m}$ for some $a, b \in \QQ$. But then $m - 4 = a^2 + mb^2 + 2ab\sqrt{m}$ and thus $a = 0$ or $b = 0$ since $m \notin \QQ^{*2}$. It is now an elementary check that $a = 0$ would imply $m(m-4) \in \QQ^{*2}$ while $b = 0$ would imply $m - 4 \in \QQ^{*2}$, both of which contradict the assumption $m, m - 4, m(m-4) \notin \QQ^{*2}$.

  Assume now that $[\QQ(\sqrt{m}, \sqrt{m - 4}) : \QQ] = 4$. It is clear that $m, m - 4 \notin \QQ^{*2}$. Moreover, the assumption implies that $m - 4 \notin \QQ(\sqrt{m})^{*2}$. If $m(m -4) \in \QQ^{*2}$, then clearly $m - 4 = (b/m)^2m$ for some $b \in \QQ^*$ and thus $m - 4 \in \QQ(\sqrt{m})^{*2}$, a contradiction.
\end{proof}

The relation between the values of the local invariant maps of $\alpha$, $\alpha_{i, -}$ and the corresponding Hilbert symbols will be of great use to us. It is given by
\begin{equation}
  \label{eqn:invp-hilbert-symbol}
  \begin{split}
    \inv_p \alpha_{i, -} (x_0, x_1, x_2, x_3) 
    &= \frac{1 - (x_i/x_0 - 2, m - 4)_p}{4}, \quad i =1,2,3, \\
    \inv_p \alpha (x_0, x_1, x_2, x_3) 
    &= \frac{1 - (x_i^2/x_0^2 - 4, m - 4)_p}{4}, \quad i =1,2,3,
  \end{split}
\end{equation}
Let $\(\frac{\cdot}{\cdot}\)$ be the Legendre symbol. If $a, b \in \ZZ$ and $\mu, \eta \in \Zp^{\times}$, the explicit formulae for the Hilbert symbol \cite[Thm.~1. p.20]{serrearithmetic} for any finite prime $p$ are given by
\begin{equation}
  \label{eqn:hilbert-symbol}
  \begin{split}
    (p^a\mu, p^b \eta)_p &=
    (-1)^{\frac{p - 1}{2}ab} \(\frac{\mu}{p}\)^b \(\frac{\eta}{p}\)^a \quad \text{if $p$ is odd and}\\
    (p^a\mu, p^b \eta)_2 &=
    (-1)^{\frac{\mu - 1}{2}\frac{\eta - 1}{2} + b\frac{\mu^2 - 1}{8} + a\frac{\eta^2 - 1}{8}}.
  \end{split}
\end{equation}

We proceed with an in-depth analysis of the values of the local invariant maps of $\alpha_{i, -}$ at all places of $\QQ$.

\begin{lemma}
\label{lem:inv-at-infinity}
  \text{ }
  \begin{enumerate}[label=\emph{(\roman*)}]
    \item If $m > 4$, then $\inv_{\infty} \alpha_{i,-}(\cdot)$ vanishes identically on $(\sX_m, \sDomega)_{\str}^*(\RR)$, $i = 1,2,3$. 
    \item If $m < 4$, $m \neq 0$, then $\inv_{\infty} \alpha_{i,-}(\cdot): (\sX_m, \sDomega)_{\str}^*(\RR)  \rightarrow \{0, 1/2\}$ surjects, $i = 1,2,3$. 
  \end{enumerate}
\end{lemma}

\begin{proof}
  The proof of (i) is clear. Indeed, $m - 4 > 0$ is a square in $\RR^*$ and the claim follows from \cite[Thm.~1, p.20]{serrearithmetic} and \eqref{eqn:invp-hilbert-symbol}.

  Assume now that $m < 4$ is non-zero. Fix $u_3 = 4$. Then $u_1^2 + u_2^2 - 4u_1u_2$ is indefinite as it takes the shape $x^2 - 3y^2$ under the linear change $u_2 - 2u_1 = x$, $u_1 = y$. Then \eqref{eqn:Um} becomes
  \[
    x^2 - 3y^2 = m - 16,
  \]
  which has the obvious $\RR$-points $(x, y) = (0, \pm \sqrt{(16 - m)/3})$. As $m < 4$ we clearly have $\sqrt{(16 - m)/3} > 2$ thus producing $M_{\infty}, N_{\infty} \in U_m(\RR)$ with $u_1 > 2$ and $u_1 < -2$, respectively. We may now apply \cite[Thm.~1, p.20]{serrearithmetic} and \eqref{eqn:invp-hilbert-symbol} to see that $\inv_{\infty} \alpha_{1, -}(M_{\infty}) = 0$ and $\inv_{\infty} \alpha_{1, -}(N_{\infty}) = 1/2$. The claim for $\alpha_{i, -}$, $i = 2,3 $ now follows from an application of an automorphism of $U_m$ that swaps $u_1$ with $u_i$. This completes the proof.  
\end{proof}

\begin{proposition} 
  \label{prop:inv_p}
  Let $p$ be a finite prime and assume that at least one of $\omega_1$, $\omega_2$, $\omega_3$ is finite. Then there exists $M_p \in (\sX_m, \sDomega)_{\str}^*(\Zp) \setminus \sU_m(\Zp)$, such that 
  \[
    \begin{split}
      \inv_p \alpha_{1, -}(M_p) = 0, \quad 
      \inv_p \alpha_{2, -}(M_p) = 0, \quad 
      \inv_p \alpha_{3, -}(M_p) = 0.
    \end{split}
  \]
  Moreover, if $\omega_i$ is finite and $p$ divides $m - 4$ to odd multiplicity, then there is another point $N_p \in (\sX_m, \sDomega)_{\str}^*(\Zp) \setminus \sU_m(\Zp)$, for which
  \[
    \inv_p \alpha_{i, -} (N_p)  = 1/2.
  \]
\end{proposition}

\begin{proof}
  Without loss of generality we may assume that $\omega_1$ is finite. We claim that for a suitably chosen $v \in \Zp^{\times}$ there exist
  \[
    M_p, N_p \in X_m(\Qp) \text{ of the shape } 
    \begin{cases}
      (p^{2\omega_1}, up^{2\omega_1}, v, 1)  &\text{with } u \in \Zp \text{ if } p > 2, \\
      (2^{2\omega_1}, u2^{2\omega_1 + 1}, v, 1)  &\text{with } u \in \ZZ_2^{\times} \text{ if } p = 2.
    \end{cases}
  \]
  Any such point belongs to $(\sX_m, \sDomega)_{\str}^*(\Zp)$ by Proposition~\ref{prop:semi-integral-criterion} but clearly is not in $\sU_m(\Zp)$. The fixed value of $v$ will guarantee that the assumption on the local invariant maps at $M_p$ and $N_p$ are fulfilled. 

  For any odd $p$ substitute $(p^{2\omega_1}, up^{2\omega_1}, v, 1)$ in \eqref{eqn:Xm}. Dividing through $p^{2\omega_1}$ then gives
  \begin{equation}
    \label{eqn:solutions-mod-p}
    p^{4\omega_1}u^2 + v^2 + 1 - uv - mp^{4\omega_1} = 0.
  \end{equation}
  For $M_p$ one may take $v = 1$ and for $N_p$ we fix $v \in \Zp^{\times} \setminus \Zp^{\times 2}$. The equation \eqref{eqn:solutions-mod-p} mod $p$ is solvable with $u \equiv v + v^{-1} \bmod p$. Since $v \in \Zp^{\times}$ the partial derivative with respect to $u$ does not vanish mod $p$ and therefore Hensel's lemma is applicable. It gives a unique lift in $ \Zp$ of $u \bmod p$ solving \eqref{eqn:solutions-mod-p} and thus shows the existence of $M_p$ and $N_p \in X_m(\Qp)$ of the desired shape.

  If $p = 2$, the argument is identical. 
  We take $v = 1$ for $M_p$ and $v = 5$ for $N_p$. Then we substitute $(2^{2\omega_1}, u2^{2\omega_1 + 1}, v, 1)$ in \eqref{eqn:Xm} and divide through $2^{2\omega_1}$. This gives
  \[
    2^{4\omega_1 + 2}u^2 + v^2 + 1 - 2uv - m2^{4\omega_1} = 0.
  \]
  As $v^2 \equiv 1 \bmod 8$, this is soluble mod 8 under the assumption $uv \equiv 1 \bmod 4$. This time the partial derivative with respect to $u$ has a $2$-adic valuation equal to 1 and thus Hensel's lemma is once again applicable, therefore verifying the existence of $2$-adic semi-integral points $M_p, N_p$ of the claimed shape.

  We continue our analysis with no restriction on $p$ other than $p < \infty$. The values of the local invariant maps follow from \eqref{eqn:invp-hilbert-symbol} and \eqref{eqn:hilbert-symbol}. At $M_p$ they are
  \[
    \inv_p \alpha(M_p) 
    = \inv_p \alpha_{2, -}(M_p) 
    = \inv_p \alpha_{3, -}(M_p) 
    = \frac{1 - (1, m - 4)_p}{4}
    = 0.
  \]
  In view of \eqref{eqn:alpha-alpha_i} and the values of the local invariant map of $\alpha$, $\alpha_{2, -}$ and $\alpha_{3, -}$ at $M_p$, we conclude that $\inv_p \alpha_{1, -}(M_p) = 0$. The claimed values of the local invariant maps at $N_p$ follow identically since 
  \[
    \begin{split}
      \inv_p \alpha_{2, -}(N_p) = 1/2 \quad \text{and} \quad 
      \inv_p \alpha(N_p) = \inv_p \alpha_{3, -}(N_p) = 0. 
    \end{split}
  \]
  This completes the proof.
\end{proof}

Recall that $\Br(X_m, \Domega)$ was defined in \cite{MNS22} as
\[
  \Br(X_m, \Domega)
  = \{\beta \in \Br U_m \ : \ \omega_i \partial_{D_{i,\fin}}(\beta)=0 \text{ when } \omega_i \neq \infty\}.
\]

\begin{lemma}
  \label{lem:alpha_i-Br-X-D}
  If $\omega_i$ is the only finite weight, then $\alpha_{i, -} \in \Br (X_m \setminus D_{\inf})$. If, moreover, $\omega_i$ is even, then also $\alpha_{i, -} \in \Br (X_m, \Domega)$.
\end{lemma}

\begin{proof}
  We have $\Br U_m/\Br \QQ = \langle \alpha_{1,-}, \alpha_{2,-}, \alpha_{3,-} \rangle$ by \cite[Prop.~4.5]{LM21}. As $D_{\inf} = D_j \cup D_k$, $j, k$ distinct in $\{1, 2, 3\} \setminus \{i\}$, we have $U_m = (X_m \setminus D_{\inf}) \setminus D_i$. Then purity \cite[Thm~3.7.1]{colliotsko_2021} gives
  \[
    \Br (X_m \setminus D_{\inf}) \rightarrow \Br U_m \rightarrow H^1(\QQ(D_i), \QQ / \ZZ).
  \]  
  The fact that $\alpha_{i, -}$ has a trivial residue along $D_i$ is explained in the proof of \cite[Prop.~4.5]{LM21}. Hence $\alpha_{i, -}$ belongs to $\Br (X_m \setminus D_{\inf})$. As $\alpha_{i, -} \in \Br U_m$ is of order 2, it belongs to $\Br (X_m, \Domega)$ by definition if $\omega_i$ is even. This completes the proof.
\end{proof}

\begin{remark}
  \label{rem:si-Br-comparison}
  The above proof combined with the proof of \cite[Prop.~2.4]{LM21} shows, in fact, that $\Br (X_m \setminus D_{\inf}) / \Br \QQ \cong (\ZZ/2\ZZ)^2 = \langle \alpha_{i,-}, \alpha_{i,+} \rangle$, where $\alpha_{i,+} = (x_i/x_0 + 2, m - 4)$. Additionally, if \eqref{eqn:star} holds, then $\Br (X_m \setminus D_{\inf}) = \Br (X_m, \Domega)$ in the case of even $\omega_i$.
\end{remark}

\begin{proposition}
  \label{prop:BM-set-alpha_i}
  If $m - 4 \notin \QQ^{*2}$ and $\omega_i$ is finite, then 
  \[
    (\sX_m, \sDomega)_{\str}^*(\AA_\QQ)^{\alpha_{i, -}} 
    \subsetneq (\sX_m, \sDomega)_{\str}^*(\AA_\QQ).
  \]
\end{proposition}

\begin{proof}
  If $m - 4 \notin \QQ^{*2}$, then either $m - 4 < 0$ or else $m - 4 > 0$ and there exists a prime $q$ such that $\val_q(m - 4)$ is odd. If $m - 4 < 0$ we can apply Lemma~\ref{lem:inv-at-infinity} to find $N_\infty \in (\sX_m, \sDomega)_{\str}^*(\RR)$ with $\inv_{\infty} \alpha_{i, -}(N_\infty) = 1/2$. On the other hand, if $\val_q(m - 4)$ is odd, there is $N_q \in (\sX_m, \sDomega)_{\str}^*(\ZZ_q)$ with $\inv_q \alpha_{i, -}(N_q) = 1/2$ by Proposition~\ref{prop:inv_p}. For each remaining place of $\QQ$ there is $M_p \in (\sX_m, \sDomega)_{\str}^*(\Zp)$ such that $\inv_p \alpha_{i, -}(M_p) = 0$ by Proposition~\ref{prop:inv_p}. Thus we get an adele whose sum of local invariant maps for $\alpha_{i, -}$ is 1/2, which shows the claim.
\end{proof}

\begin{remark}
  \label{rem:m}
  Observe that $m - 4 \in \QQ(\sqrt{m})^{*2}$ is equivalent to $m - 4 \in \QQ^{*2}$. Any element of $\QQ(\sqrt{m})^{*2}$ is either in $\QQ^{*2}$ or of the shape $a^2m/b^2$ for coprime $a, b \in \ZZ_{> 0}$. Rewriting $m - 4 = a^2m/b^2$ gives $m(b - a)(b + a) = 4b^2$. As $(a, b) = 1$, then $b^2$ divides $m$ and thus $(m/b^2)(b - a)(b + a) = 4$. This yields $a = 0$. 
\end{remark}

We are now ready to prove our main results on local-global principles. 

\begin{proof}[Proof of Theorem~\ref{thm:wa-Xm}]
  Assume that $m - 4 \in \QQ^{*2}$. Then $X_m$ is rational by \cite[Lem.~3.3.]{LM21}. Weak approximation is a birational invariant of smooth projective varieties and therefore $X_m$ satisfies that property.

  Assume now that $m - 4 \notin \QQ^{*2}$ and thus $X_m$ is not rational by \cite[Lem.~3.3.]{LM21} and Remark~\ref{rem:m}. If $m - 4 < 0$, then in view of \eqref{eqn:alpha-alpha_i} and Lemma~\ref{lem:inv-at-infinity} we may find $M_\infty, N_\infty \in X_m(\RR)$ such that $\inv_{\infty} \alpha(M_\infty) = 0$ and $\inv_{\infty} \alpha(N_\infty) = 1/2$. On the other hand, if $m - 4 > 0$, then there exists a prime $p$ such that $\valp(m - 4)$ is odd. For such $p$ there is $M_p \in X_m(\Qp)$ with $\inv_p \alpha(M_p) = 0$ by Proposition~\ref{prop:inv_p}. There is also $N_p \in X_m(\Qp)$ with $\inv_p \alpha(N_p) = 1/2$. This follows from \cite[Prop.~5.5]{LM21} for $p > 5$, from \cite[Prop.~5.7]{LM21}, whose proof holds in the general case of $m - 4$ with odd $p$-adic valuation, if $p = 3, 5$ and from \cite[Lem.~5.8]{LM21} for $p = 2$. Thus, the local invariant map of the generator $\alpha$ of $\Br X_m/ \Br \QQ$ surjects at $p$ and there is a Brauer--Manin obstruction to weak approximation.
\end{proof}

\begin{proof}[Proof of Theorem~\ref{thm:wa}]
  Part (i) follows from Theorem~\ref{thm:wwa-general-pairs}. 

  To see (ii), we may assume without loss of generality that $\omega_1$ is finite and $\omega_2, \omega_3$ are infinite. Then $D_{\fin} = D_{1, \fin}$ is isomorphic to $\mathbb{G}_m$ and satisfies weak approximation for rational points as it is open in $\PP^1$ and $\PP^1$ satisfies that property. If $m - 4 \in \QQ^{*2}$ so does $X_m$ by Theorem~\ref{thm:wa-Xm} and as $U_m$ is open in $X_m$ we conclude that $U_m$ satisfies weak approximation for rational points. The semi-integral conditions are open conditions and thus both strict and non-strict semi-integral weak approximation hold, giving the claim for $m - 4 \in \QQ^{*2}$.

  We claim that there is a Brauer--Manin obstruction to strict semi-integral weak approximation if $m - 4 \notin \QQ^{*2}$. This is indeed the case since the closure of $(\sX_m, \sDomega)_{\str}^*(\AA_\QQ)^{\Br}$ in the product topology is a strict subset of the adelic space. The later follows from \cite[Prop.~3.19(ii)]{MNS22} and Proposition~\ref{prop:BM-set-alpha_i} provided that $\alpha_{i, -}$ belongs to $\Br (X_m \setminus D_{\inf})$ in the Campana case or $\alpha$ belongs to $\Br (X_m, \Domega)$ in the Darmon case. The last claim is shown in Lemma~\ref{lem:alpha_i-Br-X-D}, which completes the proof.
\end{proof}

\begin{proof}[Proof of Theorem~\ref{thm:sa}]
  We begin by showing that all elements of $\Br U_m$ pair to zero with $(\sX_m, \sDomega)_{\str}^*(\Zp)$ for all $p \in S(m)$. Indeed, $m - 4 \in \Qp^{*2}$ and thus $\Br (U_m \times_{\QQ} \Qp) / \Br \Qp$ is trivial which verifies that claim. As $T \subset S(m)$ the strict inclusion of adeles that pair to zero with $\alpha_{i, -}$ in the adelic space shown in Proposition~\ref{prop:BM-set-alpha_i} is preserved on projection away from $T$, hence strong approximation off $T$ fails. 
\end{proof}

\begin{proof}[Proof of Theorem~\ref{thm:BMO-hp}]
  The lack of algebraic Brauer--Manin obstruction follows from Proposition~\ref{prop:inv_p} as we can construct $(M_p) \in (\sX_m, \sDomega)_{\str}^*(\AA_\QQ)$ with zero local invariant maps for all generators of $\Br U_m / \Br \QQ$ at each $p \in \Omega_\QQ$. Thus $(M_p)$ belongs to the Brauer--Manin set verifying its non-emptiness. The fact that the transcendental Brauer group of $U_m$ is trivial under the arithmetic condition on $m$ given in the statement of Theorem~\ref{thm:BMO-hp} follows from \cite[Cor.~4.3]{LM21}.
\end{proof}

\section{Existence of semi-integral points on Markoff pairs}
\label{sec:quad-forms}
We will prove in this section our main results on the existence of families of Markoff orbifold pairs with strict semi-integral points for arbitrary weights which do not come from integral points. The main idea is to use the theory of binary quadratic forms over $\ZZ$ in connection with classical results from algebraic number theory. We refer to \cite{Cas78} for details on the theory of binary forms and their genus theory. For $D\in \ZZ, D \equiv 0,1 \bmod 4$ let $\Cl(D)$ be the class group of integral binary quadratic forms of discriminant $D$ given by $\GL_2(\ZZ)$-equivalence, and let $\Cl^+(D)$ for the corresponding narrow form class group given by $\SL_2(\ZZ)$-equivalence. Furthermore, let $\sG(D) \coloneq \Cl^+(D)/\Cl^+(D)^2$ be the associated genus group. Classical results imply that all these groups are always finite. Furthermore, for a fundamental discriminant $D$ the group $\Cl^+(D)$ is isomorphic to the narrow class group $\Cl^+(\QQ(\sqrt{D}))$ (See for example \cite[Thm~6.20]{buell_1989}). 

\begin{proposition}
  \label{prop:st-semi-integral-points}
  Assume that one of $\omega_i$ is finite. If $m = 6d$ with $d \equiv 29 \bmod 78$ and the reduction mod 13 of each prime divisor of $m - 121$ lies in $R = \{1, 3, 4, 9, 10, 12\} = (\ZZ/13\ZZ)^{\times 2} \subset (\ZZ/13\ZZ)^\times$, then $(\sX_m,\sDomega)_{\str}^*(\ZZ) \neq \emptyset$ but $\sU_m(\ZZ) = \emptyset$.
\end{proposition}

\begin{proof}
  We may assume that $\omega_1 < \infty$. Under our assumptions $m \equiv 3 \bmod 9$ and therefore $\sU_m(\ZZ) = \emptyset$ by \cite[Prop.~6.1]{Ghosh2017IntegralPO}. In view of Proposition~\ref{prop:semi-integral-criterion-2-inf} and Remark~\ref{rem:inclusion-and-weights}, we shall be looking at $M \in (\sX_m,\sDomega)_{\str}^*(\ZZ) \setminus \sU_m(\ZZ)$ of the shape $(3^{\omega_1}, 11 \cdot 3^{\omega_1}, y, z)$ with $3 \nmid yz$. Let $n = (m - 121)3^{2\omega_1}$. Substituting $M$ in \eqref{eqn:Xm} and dividing through $3^{\omega_1}$ gives
  \[
    y^2 - 11yz + z^2 = n.
  \]

  We claim that $f(y,z) = y^2 - 11yz + z^2$ of discriminant $117$ represents $n$ primitively over $\ZZ$. The indefinite form $f$ represents $n$ over $\RR$. On the other hand, the congruence condition on $d$ guaranties the existence of a primitive representation of $n$ over $\ZZ_2$, $\ZZ_3$ and $\ZZ_{13}$, while so does the restriction on the prime divisors of $n$ over $\ZZ_p$ for all remaining primes dividing $n$. We may now apply \cite[Thm~5.1, p.~143]{Cas78}, which confirms that some form in the genus of $f$ represents $n$ primitively over $\ZZ$. The total number of genera $\#\sG(117)$ is equal to two, hence $\sG(117) \cong \Cl^+(117)$. Therefore, all forms in the genus of $f$ are properly equivalent to each other, and $f$ properly represents $n$, as claimed. Finally, it remains to show that $3 \nmid yz$. Indeed, if $3 \mid y$, then as $3 \mid n$ we must have $3 \mid z$, which contradicts the primitiveness of the representation. Similarly, if $3 \mid z$. This concludes the proof.
\end{proof}

We are now in position to prove Theorem~\ref{thm:count}. To do so it suffices to evaluate $S(B)$, which counts the number of $m$ as in Proposition~\ref{prop:st-semi-integral-points}, whose absolute value is at most $B$. 

\begin{proof}[Proof of Theorem~\ref{thm:count}]
  Let $S(B)$ be as above. We claim that there is a real constant $c > 0$, such that
  \begin{equation}
    \label{eqn:S(B)}
    S(B) = c\frac{B}{(\log B)^{1/2}} + O\(\frac{B}{\log B}\).
  \end{equation}

  To simplify what follows, let
  \[
    \rho(n) =
    \begin{cases}
      1 &\text{if }  p \mid n \implies p \bmod 13 \in R, \\
      0 &\text{otherwise}.
    \end{cases} 
  \]
  Let $n = 6d - 121$. This is equivalent to $n \equiv 5 \bmod 6$. The conditions of Proposition~\ref{prop:st-semi-integral-points} imply that all prime divisors of $n$ must lie in $R$ after reduction mod 13 and $n = 6d - 121 \equiv 53 \bmod 78$. The latter congruence clearly implies $n \equiv 5 \bmod 6$ and therefore
  \[
    S(B)
    = \sum_{\substack{n \le B - 121 \\ n \equiv 53 \bmod 78}} \rho(n)
    = \frac{1}{24} \sum_{\chi \bmod 78} \chi(53) T_{\chi}(B)
  \]
  by the orthogonality of Dirichlet characters mod 78 and the fact that $53$ is its own inverse mod 78. Here for a Dirichlet character $\chi$ mod 78 we have defined
  \[
    T_{\chi}(B)
    =  \sum_{n \le B - 121} \rho(n)\chi(n).
  \]
  
  The treatment of $T_{\chi}(B)$ can be done using the Landau-Selberg-Delange method. Let $\chi_0$ be the trivial Dirichlet character mod 78 and denote by $\chi_1$ the Dirichlet character mod 78 given by $\(\frac{\cdot}{13}\)\chi_0(\cdot)$. We claim that there is another real constant $c' > 0$, such that
  \begin{equation}
    \label{eqn:T_chi}
    T_{\chi}(B) =
      \begin{cases}
      c' B(\log B)^{-1/2} + O\(B(\log B)^{-1}\) &\text{if } \chi = \chi_0 \text{ or } \chi = \chi_1, \\
      O\(B(\log B)^{-1}\)  &\text{otherwise.}
      \end{cases}
  \end{equation}
  This verifies \eqref{eqn:S(B)} with $c = c'/12$ as $\chi_1(53) = 1$.

  Let $F(s, \chi)$ be the Dirichlet series corresponding to $T_{\chi}(B)$, which is well-defined for $\sigma > 1$  under the standard notation $s = \sigma + i\tau$. Define $\psi(\cdot) = \(\frac{\cdot}{13}\)\chi(\cdot)$. Since $R = (\ZZ/13\ZZ)^{\times 2} \subset (\ZZ/13\ZZ)^\times$, by the binomial series expansion we conclude that $F(s, \chi)$ is given by
  \[
    \begin{split}
      F(s, \chi) 
      &= \sum_{n = 1}^{\infty} \frac{\rho(n)\chi(n)}{n^s}
      = \prod_{p \bmod 13 \in R} \(1 - \frac{\chi(p)}{p^s}\)^{-1} \\
      &= \prod_{p} \(1 - \frac{1}{2}\(1 + \(\frac{p}{13}\)\)\frac{\chi(p)}{p^s}\)^{-1} \\
      &= \prod_{p} \(1 - \frac{\chi(p)}{p^s}\)^{-1/2}\(1 - \frac{\psi(p)}{p^s} \)^{-1/2} E_p(s)
      = L(s, \chi)^{1/2} L(s, \psi)^{1/2}E(s), 
    \end{split}
  \]
  where the function $E(s) = \prod_{p} E_p(s)$ satisfies $E_p(s) = 1 + O(p^{-2\sigma})$.  

  Note that $|\rho(n)\chi(n)| \le \rho(n)$ and $L(s, \chi_0) = \zeta(s)(1 - 1/2^s)(1 - 1/3^s)(1 - 1/13^s)$. It is now clear that $F(s, \chi)$ satisfies the hypothesis of \cite[Thm.~II.5.2]{Ten15} with $N = 0$, and with $z = w = 1/2$ if $\chi = \chi_0, \chi_1$ and $z = 0$, $w = 1/2$ for the remaining characters mod 78. Indeed, this is verified by \cite[Thm.~11.3, p.~360]{MV07} and \cite[Thm.~11.4, p.~362]{MV07} as $L(s, \chi_0 \(\frac{\cdot}{13}\))$ and $L(s, \(\frac{\cdot}{13}\))$ have no Siegel zeroes (eg. LMFDB), where $\(\frac{\cdot}{13}\)$ is the quadratic character mod 13. This confirms \eqref{eqn:T_chi} and completes the proof of Theorem~\ref{thm:count}.
\end{proof}

Finally, we will show that the above idea also adapts to other families of Markoff orbifold pairs. As a corollary we provide more examples of pairs for which there are strict semi-integral points but no integral points.

\begin{theorem}
  \label{thm:existencepcases}
  Let $p>2$ be a prime such that $\Cl(p^2+4)$ is trivial. Assume that one of $\omega_i$ is finite. If $m$ is of the form $(2+p^2)^2\pm p^{\alpha}k^2$ for some integers $\alpha \geq 0, k>0$ with $p \nmid k$, then $(\sX_m,\sD_{\uomega})_{\str}^{*}(\ZZ ) \neq \emptyset$.
\end{theorem}

\begin{proof}
  We will show the existence of semi-integral points of the shape $(p^{\omega_1},(2+p^2)p^{\omega_1},kx,ky)$ with $p\nmid kxy$. These points belong to $(\sX_m,\sD_{\uomega})_{\str}^{*}(\ZZ )$ by Proposition~\ref{prop:semi-integral-criterion}. Substituting such a point in \eqref{eqn:Xm} and dividing by $p^{\omega_1}$ yields
  \[
    x^2-(2+p^2)xy+y^2 = \pm p^{2\omega_1+\alpha}.
  \]

  Let $f(x,y)=x^2-(2+p^2)xy+y^2$ be the binary quadratic form defined by the left hand side. We will show that $f$ is properly equivalent to a quadratic form, which represents any $\pm p^l$ with $l \in \ZZ_{\ge 4}$. The transformation $x \mapsto -z,y\mapsto w-z$ is proper and maps $f$ to $g(w,z)= w^2+p^2w z -p^2z^2$. It is clear, that any representation $g(w,z)=\pm p^l$ must satisfy $p\mid w$. If $w= pt$ it then suffices to show that $h(t,z)=t^2+pt z -z^2$ properly represents $\pm p^{l-2}$. The form $h$ properly represents $p^{l-2}$ if and only if it properly represents $-p^{l-2}$, which may be seen for example from the proper transformation $t \mapsto z, z\mapsto -t$. The discriminant of $h$ reduces is a square mod $p$. Applying Hensel's lemma now shows that $\sqrt{p^2+4}\in \ZZ_p$, and therefore the congruence $t^2\equiv p^2+4 \pmod{p^{l-2}}$ is soluble. By \cite[Lemma 2.5]{cox_2013} there exists a form of discriminant $p^2+4$ that properly represents $p^{l-2}$. This form is primitive since the discriminant is coprime to $p$. By assumption $\Cl(p^2+4)$ is trivial. Moreover, the narrow form class group $\Cl^+(p^2+4))$ is then also trivial since the fundamental unit $1/2(p-\sqrt{p^2+4}) \in \QQ(\sqrt{p^2+4})$ is of norm $-1$. Therefore, all primitive forms of discriminant $p^2+4$ are properly equivalent to each other and thus $h$ properly represents $p^{l-2}$. Finally, because of the shape of $h$ and any proper representation of $p^{l-2}$ obeys $p \nmid tz$. Then $p \nmid xy$ as $x = -z$ and $y=pt -z$, which concludes the proof. 
  \end{proof}

\begin{remark}
 A quick computation reveals that the primes $2<p<1000$ for which $\Cl(p^2+4)$ is trivial are given by 3, 5, 7, 11, 13, 17.
\end{remark}

\begin{example}
 By the methods of reduction theory developed in \cite{Ghosh2017IntegralPO}, the authors showed that $m=46$ is the first positive integral Hasse failure, i.e. $\sU_{46}(\ZZ_p) \neq \emptyset$ for all $p$, but $\sU_{46}(\ZZ)=\emptyset$ . Notice that we can write $46 = (2+3^2)^2-3^15^2$, and $\Cl(3^2+4)$ is trivial. So by Theorem~\ref{thm:existencepcases} there exist semi-integral points on $(\sX_{46},\sD_{\uomega})_{\str}^{*}(\ZZ )$ for arbitrary weights.  
\end{example}

The following theorem investigates the existence of semi-integral points on a specific family of Markoff surfaces given in \cite[Prop.5.12 (iv)]{colxu_2020}, which has been shown to have no integral points. 

\begin{theorem}
  \label{thm:existence-CTXW}
 Let $m=4-3l^2$ for some prime $l\geq 17$. Assume that one of the $\omega_i$ is finite. If $-3$ is a square mod $84+l^2$, 
then $(\sX_m,\sD_{\uomega})_{\str}^{*}(\ZZ ) \neq \emptyset$.
\end{theorem}

\begin{proof}
 Without loss of generality we may assume that $\omega_1$ is finite. We want to show the existence of semi-integral points of the shape $\(3^{\omega_1},16\cdot 3^{\omega_1},x,y\)$ with $3 \nmid xy.$ Plugging points of this form in \eqref{eqn:Xm} yields
\[
 x^2-16xy+y^2 = -3^{2\omega_1+1}(84+l^2).
\]
Thus, one may conclude the proof by showing that the form $f(x,y)=x^2-16xy+y^2$ properly represents the RHS of the above equation. Now, $f$ is properly equivalent to $x^2-63y^2$, and thus by similar manipulations as in the proof Theorem~\ref{thm:existencepcases}, it suffices to show that the form $h(x,y)=x^2-7y^2$ properly represents $-3^{2\omega_1-1}(84+l^2)$. The discriminant of $h$ is a quadratic residue modulo $3$, and also modulo $84+l^2$ because $28 \equiv x^2 \pmod{84+l^2}$ has a solution if and only if $-3 \equiv x^2 \pmod{84+l^2}$ has a solution which is one of our assumptions. Since $\gcd(3,84+l^2)=1$, again by \cite[Lemma 2.5]{cox_2013} we find that there exists some form of discriminant $28$ that represents $-3^{2\omega_1-1}(84+l^2)$. In order to conclude that this form is in the same class as in $\Cl^+(28)$, we notice that $-3$ is a quadratic residue mod $28$, and therefore $-3^{2\omega_1-1}(84+l^2)$ is in the subgroup of squares in $(\ZZ/28\ZZ)^*$. Thus, any form representing this value is in the principal genus $\Cl^+(28)^2$. Since $|\Cl^+(28)|=2$, comparing the orders we may conclude that all forms in the principal genus $\Cl^+(28)^2$ are properly equivalent to each other, hence proving our claim.
\end{proof}

\bibliographystyle{amsalpha}{}
\bibliography{bibliography}
\end{document}